\documentclass{article}
\usepackage{amsrefs}
\usepackage{amssymb}
\usepackage{amsmath}
\usepackage{amsthm}
\usepackage{graphicx}

\newtheorem{theorem}{Theorem}
\newtheorem{proposition}{Proposition}
\newtheorem{lemma}{Lemma}

\DeclareMathOperator{\Pic}{Pic}
\DeclareMathOperator{\Kar}{char}
\newtheorem*{subject}{2000 Mathematics Subject Classification}
\newtheorem*{keywords}{Keywords}

\author{Marc Coppens\footnote{KU  Leuven, Technologiecampus Geel, Departement Elektrotechniek (ESAT),
Kleinhoefstraat 4, B-2440 Geel, Belgium; email: marc.coppens@kuleuven.be.}}
\title{The scrollar invariants of $k$-gonal curves having a nodal model on a smooth quadric having its nodes on few lines }
\date{}

\begin{document}
\maketitle \noindent

\begin{abstract}
We determine the scrollar invariants of the normalization $C$ of a nodal curve $\Gamma$ of type $(k,a)$ on a smooth quadric $\mathbb{P}^1 \times \mathbb{P}^1$ associated to the $g^1_k$ defined by the pencil of lines of type $(0,1)$ in case all nodes are contained in at most $k-1$ lines of type $(1,0)$.
This result is very much related to results obtained in \cite{ref1}, but the proof follows directly from an easy lemma not mentioned in \cite{ref1}.
Also the main theorem in \cite{ref1} is a consequence of that lemma making the arguments much shorter.
\end{abstract}

\begin{subject}
14H51
\end{subject}

\begin{keywords}
gonality, scrollar invariants, curves on quadrics
\end{keywords}

\section{Introduction}\label{section1}

Let $C$ be an irreducible smooth complete curve of genus $g$ defined over an algebraically closed field $K$ and having a complete base point free linear system $g^1_k$ (we fix such $g^1_k$ on $C$ from now on).
The linear system $g^1_k$ has scrollar invariants $0 \leq e_1 \leq e_2 \leq ... \leq e_{k-1}$: the integer $e_i+2$ is the smallest integer $n$ such that $h^0(C,ng^1_k)-h^0(C,(n-1)g^1_k)>i$.
The Riemann-Roch Theorem implies this inequality is equivalent to $h^0(C,\omega _C-(n-1)g^1_k)-h^0(C,\omega _C-ng^1_k)<k-i$.
Those scrollar invariants satisfy the equality $e_1+ ... +e_{k-1}=g-k+1$.

Let $M=\mathbb{P}^1 \times \mathbb{P}^1$ and let $\Gamma$ be an irreducible divisor of type $(k,a)$ on $M$ (here $k,a>0$).
Let $C$ be the normalization of $\Gamma$.
The pencil of lines of type $(0,1)$ induce a base point free linear system $g^1_k$ on $C$.
In our main result we determine the scrollar invariants of $g^1_k$ in case $\Gamma$ is a nodal curve and all nodes are contained on at most $k-1$ lines of the  pencil of lines of type $(1,0)$ on $M$.

This main result is strongly inspired by the results obtained in \cite{ref1}.
In part (i) of the proof of the main theorem in \cite{ref1} our main result is obtained in the case both the choices of the lines of type $(1,0)$ and the position of the nodes of $\Gamma$ on those lines are general.
Our main result implies this generality condition is not necessary.
Moreover the proof of that part of the main theorem in \cite{ref1} at the end consists of a reference to two remarks and part of the proof of a claim in Lemma 1 in \cite{ref1}.
The proof of that claim makes intensively use of the method d'Horace.
It seems to me that in order to finish the arguments one needs the result of Lemma \ref{lemma3} in our paper.
I do not see how this follows from the comments given in \cite{ref1}.
The proof of Lemma \ref{lemma3} is very elementary and easy and does not need the method d'Horace.
This Lemma \ref{lemma3} is crucial for the proof of our main results without making use of the method d'Horace.

The main theorem in \cite{ref1} is Theorem \ref{theorem2} in our paper.
Determining which sequences $0 \leq e_1 \leq ... \leq e_{k-1}$ do occur as scrollar invariants of smooth complete curves of genus $g$ is an important problem and the main theorem of \cite{ref1} gives a very interesting result on this problem.
Therefore it seems worthwhile to have arguments how this follows from Lemma \ref{lemma3} avoiding more complicated arguments as in \cite{ref1}.
Also in Theorem \ref{theorem2} the genus bound on $g(k,e)$ is a little bit better than the one obtained in \cite{ref1}.
It is important to mention that the idea for obtaining Theorem \ref{theorem2} using nodal curves on a smooth quadric is completely coming from \cite{ref1}.

\section{Generalities}\label{section2}

Let $D$ be a divisor on a smooth complete variety $X$, then we denote $\mathcal{O}(D)$ for the associated invertible sheaf on $X$.
It defines an element $c \in \Pic (X)$.
If the space of global sections $\Gamma (C,\mathcal{O}(D))$ is non-zero we write $\vert D \vert$ to denote the associated complete linear system.
We also write $\vert c \vert$ or $\vert \mathcal{O}(D) \vert$ to denote this linear system.
In case $X$ is a curve, $\deg (D)=d$ and $\dim \vert c \vert =r$ then we say $\vert c \vert$ is a $g^r_d$ on $X$.
More general, a $g^r_d$ on a curve $X$ can be a linear subspace of dimension $r$ of some complete linear system of degree $d$ on $X$ (so it need not be complete).
If $X$ is a curve and $\vert D \vert$ is a $g^r_d$ we also write $\vert mg^r_d \vert$ to denote $\vert mD \vert$.

For any smooth complete variety $X$ we write $\omega _X$ to denote the canonical line bundle on $X$.
In case $\Omega _X$ is a canonical divisor (it need not be effective) and $D$ is a divisor on $X$ defining $c \in \Pic (X)$ then we also write $\vert \omega _X -mD \vert$ (or $\vert \omega _X -mc \vert$) to denote $\vert \Omega _X-mD \vert$.
In case $X$ is a curve and $\vert D \vert$ is a $g^r_d$ on $X$ then we also write $\vert \omega _X-mg^r_d \vert$.

Now let $M=\mathbb{P}^1 \times \mathbb{P}^1$.
There are two projections $pr_i : M \rightarrow \mathbb{P}^1$ ($i\in \{ 1,2 \}$).
For $P \in \mathbb{P}^1$ the fiber $pr_i^{-1}(P)$ is an effective divisor on $M$.
Refering to the classical imbedding $M \subset \mathbb{P}^3$ as a smooth quadric, we call it a line.
The associated element of $\Pic (M)$ is denoted by $(1,0)$ in case $i=1$ and $(0,1)$ in case $i=2$.
It is well-known that $\Pic (M)=\mathbb{Z}.(1,0) \oplus \mathbb{Z}.(0,1)$.
For $a,b \in \mathbb{Z}$ we write $(a,b)$ to denote $a.(1,0)+b.(0,1)$.
We also write $\mathcal{O}(a,b)$ to denote an invertible sheaf associated to $(a,b)\in \Pic (M)$.
The intersection number $(a,b).(a',b')$ is given by $ab'+ba'$.
One has $\vert (a,b) \vert \neq \emptyset$  if and only if $a \geq 0$ and $b \geq 0$ and in that case $\dim (\vert (a,b) \vert)=(a+1)(b+1)-1=ab+a+b$ and an element of $\vert (a,b) \vert$ is called a curve of type $(a,b)$.
One has $\omega _M= \mathcal{O} (-2,-2)$.

An important ingredient in the proofs of this paper is the following well-known lemma.

\begin{lemma}\label{lemma1}
If $a,b\geq 0$ then $H^1(\mathcal{O}(a,b))=0$.
\end{lemma}

\begin{proof}
If one uses \cite{ref2}, Chapter III, Exercise 5.6 (a)(2) then it follows from Serre duality (especially \cite{ref2}, Chapter III, Corollary 7.7).

It can also be proved directly by means  of induction starting with $H^1(\mathcal{O} _M)=0$ ($M$ is a rational surface) and using some exact cohomology sequences.

A short argument is as follows. By explicitly writing down above the two standard affine open subsets $\mathbb{A} ^1$ of $\mathbb{P} ^1$ (one omitting $0$ and one omitting $\infty$) one finds
\[
(pr_2)_*(\mathcal{O} (a,b))=\mathcal{O}(b) \oplus ... \oplus \mathcal{O}(b) \text{   ($a$ terms)}
\]
(here $\mathcal{O} (b)$ is an invertible sheaf of degree $b$ on $\mathbb{P}^1$).
From \cite{ref2}, Chapter V, Lemma 2.4 one knows $H^1(\mathcal{O}(a,b))=H^1((pr_2)_*(\mathcal{O}(a,b)))$.
On $\mathbb{P}^1$ one has $H^1(\mathcal{O}(b))=H^0(\mathcal{O}(-2-b))=0$ for $b \geq 0$.
\end{proof}

Let $Z$ be a 0-dimensional subscheme of $M$.
The set of divisors $\Gamma$ of $\vert (a,b) \vert$ containing $Z$ is a linear subspace of $\vert (a,b) \vert$ denoted by $\vert (a,b) -Z \vert$.
In case $Z_1$ and $Z_2$ are two disjoint 0-dimensional subschemes of $M$ we write $Z_1+Z_2$ to denote their union.

Now, let $\Gamma$ be an irreducible nodal curve of type $(a,b)$ on $M$ having nodes at $P_1, ... , P_m$ and let $n : C \rightarrow \Gamma$ be the normalization of  $\Gamma$.
Write $n^{-1}(P_i)=\{ P_{i1}, P_{i2} \}$.

\begin{lemma}\label{lemma2}
Using the notations and the situation mentioned above, we have
\[
\vert \omega _C \vert = \{ n^{-1} (D.\Gamma) - \sum^m_{i=1} (P_{i1}+P_{i2}) : D \in \vert (a-2,b-2)-P_1- ... - P_m \vert \} \text{.}
\]
In particular
\[
\dim ( \vert \omega _C \vert )= \dim ( \vert (a-2,b-2) -P_1- ... - P_m \vert ) \text{.}
\]
\end{lemma}

\begin{proof}
See the first part of Remark 2 in \cite{ref1}.
It also follows from the proof of Lemma 2 in \cite{ref3}.
\end{proof}

\section{Proofs}\label{section3}

All results in this paper follow from the next elementary and easy to prove lemma.
In this paper, in order to reprove the main theorem of \cite{ref1}, it replaces the intensive use of the method d'Horace.

\begin{lemma}\label{lemma3}
Fix $a \in \mathbb{Z}_{\geq 0}$.
On $M = \mathbb{P}^1 \times \mathbb{P}^1$ choose $k+1$ different lines $L_1, ... ,L_{k+1}$ of type $(1,0)$ and for $1 \leq i \leq k+1$ choose an effective divisor $D_i$ of degree $y_i$ on $L_i$ with $a \geq y_1 \geq y_2 \geq ... \geq y_{k+1} \geq 0$ (so some  divisors $D_i$ are allowed to be 0).
Then
\[
\dim ( \vert (k,a) -D_1 - ... - D_{k+1} \vert )=ka+k+a-y_1 - ... - y_{k+1} \text { .}
\]
\end{lemma}

\begin{proof}
In case $a=0$ there is nothing to prove, so we assume $a \geq 1$.

First assume $k=0$.
One has $\dim (\vert (0,a) \vert)=a$ and each element of $\vert (0,a) \vert$ is the sum of $a$ lines of type $(0,1)$.
The fact that this sum needs to contain the divisor $D_1$ on $L_1$ imposes $y_1$ independent conditions on $\vert (0,a) \vert$.
This proves the desired dimension claim in case $k=0$.

Take $k \geq 1$ and assume the dimension claim holds for all smaller values of $k$.
Let $k'= \max \{ i : y_i \neq 0 \}$ (hence $0 \leq k' \leq k+1$ with $k' = 0$ corresponding to $y_1 = 0$).
In case $k' = 0$ there is nothing to prove, so we can assume $k' \geq 1$.

First assume $k'=1$ and consider the exact sequence
\[
0 \rightarrow \mathcal{O}_M(k-1,a) \rightarrow \mathcal{O}_M (k,a) \rightarrow \mathcal{O}_{L_1}(a) \rightarrow 0 \text { .}
\]
From $h^1(\mathcal{O}_M(k-1,a))=0$ (see Lemma \ref{lemma1}) we find $\vert (k,a) \vert$ induces the complete linear system $\vert \mathcal{O}_{L_1}(a) \vert$ of degree $a$ on $L_1 \cong \mathbb{P}^1$.
We conclude again the condition $D_1 \subset \Gamma$ for $\Gamma \in \vert (k,a) \vert$ induces $y_1$ independent linear conditions on $\vert (k,a) \vert$, again proving the desired dimension claim in this case.

We are also going to use induction on $k'$.
So assume $k' \geq 2$ and the dimension claim holds for smaller values of $k'$ (for the linear system $\vert (k,a) \vert$).
Fix one more line $L_{k+2}$ of type $(1,0)$ and use the lines $L_2, L_3, ... , L_{k+1}, L_{k+2}$ and the divisor $D_{k+2}=0$ on $L_{k+2}$.
So we use $y_2 \geq y_3 \geq ... \geq y_{k+1} \geq y_{k+2}=0$.
For this situation the value of $k'$ drops by one and from the induction hypothesis we obtain
\[
\dim ( \vert (k,a)(D_2 - ... - D_{k+1} ) \vert ) = ak+a+k-y_2 - ... - y_{k+1} \text { .}
\]
Assume $\vert (k,a) (-D_2 - ... - D_{k+1}) \vert$ does not induce the complete linear system $\vert \mathcal{O}_{L_1} (a) \vert$ on $L_1$.
Then the condition $L_1 \subset \Gamma$ for $\Gamma \in \vert (k,a)(-D_2 - ... -D_{k+1} ) \vert$ induces at most $a$ independent linear conditions on $\vert (k,a)(-D_2 - ... -D_{k+1}) \vert$.
Such elements are  the sum of $L_1$ and a divisor in $\vert (k-1,a)(-D_2 - ... - D_{k+1} ) \vert$, so it implies
\[
\dim ( \vert (k-1,a) (-D_2 - ... -D_{k+1})\vert )=a(k-1)+a+(k-1)-y_2 - ... -y_{k+1}+1 \text { .}
\]
This contadicts the induction hypothesis on $k$.
Therefore $\vert (k,a) (-D_2 - ... - D_{k+1}) \vert$ induces the complete linear system $\vert \mathcal{O}_{L_1}(a) \vert$ on $L$ and since $y_1 \leq a$ it implies
\[
\dim (\vert (k,a)(-D_1 - ... - D_{k+1}) \vert ) =ak+a+k-y_1 - ... -y_{k+1} \text { .}
\]
\end{proof}

Now let $L_1, ..., L_{k-1}$ be different lines of type $(1,0)$ on $M$ and for $1 \leq i \leq k-1$ let $D_i$ be a  reduced divisor of degree $y_i$ on $L_i$ with $a \geq y_1 \geq ... \geq y_{k-1} \geq 0$ ($a$ being a fixed element of $\mathbb{Z}_{\geq 1}$).
Let $\Gamma$ be a nodal curve of type $(k,a)$ on $M$ having its set of nodes equal to $D_1 + ... + D_{k-1}$.
Let $n : C \rightarrow \Gamma$ be  the normalization of $\Gamma$.
The pencil of lines $\vert (0,1) \vert$ induces a $g^1_k$ on $C$.
In the next theorem we determine the scrollar invariants of this $g^1_k$.

As mentioned in the introduction this result is also obtained (under some generality assumptions) in \cite{ref1} (part (i) of the proof of the main theorem).
Despite the importance of the determination of those scrollar invariants, the arguments to obtain them are absent in \cite{ref1} except for some indications with no details.
We give a complete proof only refering to our Lemma \ref{lemma3}.

\begin{theorem}\label{theorem1}
Let $C$ and $g^1_k$ be as described above.
The scrollar invariants $e_i$ of $g^1_k$ are given by $e_i=a-y_i-2$.
\end{theorem}

\begin{proof}
Define $t \geq 1$ and $s_0=1 < s_1 < ... < s_t=k$ such that $y_{s_0} > y_{s_1} > ... > y_{s_{t-1}}$ with $\{ y_1, ..., y_{k-1} \} = \{ y_{s_0}, y_{s_1}, ..., y_{s_{t-1}} \}$ (so for $1 \leq i \leq t$ exactly $s_i - s_{i-1}$ integers $y_i$ are equal to $y_{s_{i-1}}$).
So we need to prove $e_1= ... = e_{s_1-1}=a-y_1-2$  (in case $t=1$ this describes all scrollar invariants); for $1 \leq l \leq t$ and $s_{l-1} \leq i <s_l$ one has $e_i = a - y_{s_{l-1}}-2$.
We are going to compute the function $f(n)=h^0(C,\omega _C - (n-1)g^1_k)-h^0(C,\omega_C -ng^1_k)$ for $n \in \mathbb{Z}_{\geq 1}$.

In case $m \leq a-2-y_1$, hence $a-2-m \geq y_1$, we have because of Lemma \ref{lemma3}
\[
\dim \vert \omega _C - mg^1_k \vert = \dim \vert (k-2,a-2-m) -D_1 - ... - D_{k-1} \vert = 
\]
\[
=(k-2)(a-2-m)+(k-2)+(a-2-m)- y_1 - ... - y_{k-1} \text { .}
\]
This implies $f(m)=k-1$.

Take $1 \leq l \leq t$ and $a-2-y_{s_{l-1}}<m \leq a-2-y_{s_l}$ (we take $y_{s_t}=0$), hence $y_{s_l} \leq a-2-m <y_{s_{l-1}}$.
For $1 \leq i \leq s_l-1$ and $\Gamma \in \vert (k-2,a-2-m)(-D_1 - ... -D_{k-1}) \vert$ we need $\Gamma \cap L_i$ contains $D_i$.
However $(k-2,a-2-m).L_i=a-2-m< y_{s_{l-1}} \leq y_i$ therefore $\Gamma$ has to contain $L_i$.
This implies
\[
\vert (k-2,a-2-m)-D_1 - ... - D_{k-1} \vert= 
\]
\[ = \vert (k-1-s_l, a-2-m) -D_{s_{l-1}} - ... - D_{k-1} \vert +L_1+ ... + L_{s_l -1} \text { .}
\]
Since $\dim ( \vert \omega _C - mg^1_k \vert )= \dim ( \vert (k-2,a-2-m) -D_1 - ... - D_{k-1} \vert )$ it follows
\[
\dim ( \vert \omega _C -mg^1_k \vert ) = \dim ( \vert (k-1-s_l,a-2-m) -D_{s_l} - ... - D_{k-1} \vert \text { .}
\]
Since $a-2-m \geq y_{s_l}$ we can use Lemma \ref{lemma3} and conclude
\[
\dim \vert \omega _C-mg^1_k \vert = (k-1-s_l)(a-2-m)+(k-1-s_l)+(a-2-m)-y_{s_l} - ... - y_{k-1} \text { .}
\]
This implies
\[
f(a-1-y_{s_{l-1}})= \dim (\vert \omega_C - (a-2-y_{s_{l-1}})g^1_k \vert) - \dim (\vert \omega_C - (a-1-y_{s_{l-1}})g^1_k \vert )=
\]
\[
= (s_l - s_{l-1})y_{s_{l-1}}+ (k-s_{l-1})-(s_l -s_{l-1})y_{s_{l-1}} = k-s_{l-1} \text { .}
\]
For $a-y_{s_{l-1}} \leq m \leq a-2-y_{s_l}$ (this case does not occur if $y_{s_l}=y_{s_{l-1}} -1$) one has $f(m) = k-s_l$.
In particular we find $f(m)=0$ for $m \geq a-y_{s_{t-1}}$.

As mentioned in the introduction, $e_i+2$ is the minimal integer $n$ such that $f(n)<k-i$.
Since $f(n)=k-1$ if and only if $n\leq a-1-y_1$ and , in case $t>1$, $f(n)=k-s_1$ if and only if $a-y_1 \leq n \leq a-1-y_{s_1}$ we obtain $e_1 = ... = e_{s_1-1} = a-y_1-2$.
In case $t=1$ we obtain $e_1 = ... = e_{k-1}=a-y_1-1$.
For $1 \leq l \leq t-1$ we have $f(a-1-y_{s_{l-1}})=k-s_{l-1}$ while $f(a-y_{s_{l-1}})=k-s_l$.
This implies $e_i=a-y_{s_{l-1}}-2$ for $s_{l-1} \leq i < s_l$.
We obtainen $f(n)=0$ if and only if $n\geq a - y_{s_{t-1}}$ while $f(a-y_{s_{t-1}}-1)=k-s_{t-1}$.
This implies $e_i = a-y_{s_{t-1}}-2$ for $s_{t-1} \leq i \leq k-1$.                                              
\end{proof}

In the next proposition we prove the existence of nodal curves on $M=\mathbb{P}^1 \times \mathbb{P}^1$ having a base point free pencil $g^1_k$ cut out by the pencil of lines $\vert (0,1) \vert$ on $M$ and such that all nodes are contained in the union of at most $k-1$ lines of type $(1,0)$ on $M$.
In \cite{ref1} this is proved in Lemma 1 and its proof is the main part of that paper.
It is rather involved and it makes intensive use of the method d'Horace.
As already mentioned at the introduction we give a proof only using Lemma \ref{lemma3} and theorems of Bertini.

\begin{proposition}\label{proposition1}
On $M=\mathbb {P}^1 \times \mathbb{P}^1$ let $L_1, ... , L_{k-1}$ be different lines of type $(1,0)$ and for each $1 \leq i \leq k-1$ choose an effective reduced divisor $D_i$ of degree $y_i$ on $L_i$ with $y_1 \geq ... \geq y_{k-1} \geq 0$.
Assume those divisors are taken such that there is no line of type $(0,1)$ containing at least two points of $D_1 + ... + D_{k-1}$.
Let $a$ be an integer with $a \geq (k-1)y_1$ in case $Kar (K)=0$ and $a \geq (k-1)y_1+1$ in case $\Kar (K) =p>0$.
There exists an irreducible nodal curve $\Gamma \in \vert (k,a) \vert$ having $D_1 + ... + D_{k-1}$ as the set of nodes.
\end{proposition}

\begin{proof}
Let $\mathcal{L}$ be the linear system $\vert (k-1,y_1)-D_1 - ... - D_{k-1} \vert$ in case $\Kar (K)=0$ and $\vert (k-1,y_1+1) - D_1 - ... - D_{k-1} \vert$ in case $\Kar (K)=p>0$.
For $1 \leq i \leq k-1$ this linear system induces by restriction to $L_i$ a linear subsystem of $D_i + \vert \mathcal{O}_{L_i}(y_1 - y_i) \vert$ in case $\Kar (K)=0$ (in case $y_1=y_i$ this linear system has dimension 0) and of $D_i + \vert \mathcal{O}_{L_i}(y_1-y_i+1) \vert$ in case $\Kar (K)=p>0$.
In case $\Kar (K)=0$ applying Lemma \ref{lemma3} to the linear system $\vert (k-2,y_1) - D_1 - ... - \widehat{D_i} - ... - D_{k-1} \vert$ (hat means omitted) we obtain equality.
A similar statement holds in case $\Kar (K)=p>0$.
This shows a general element $\gamma$ of $\mathcal{L}$ intersects each line $L_i$ ($1 \leq i \leq k-1$) transversally (in particular $\gamma$ is smooth at each point of $D_i$) and $\mathcal{L}$ has no fixed point on $L_i$ outside of $D_i$.
In case $P \in L_i \setminus D_i$ and $\Kar (K)=p>0$ it also implies there exists $\gamma \in \mathcal {L}$ smooth at $P$ and such that $T_P({\gamma}) \neq L_i$.

Take any line $L$ of type $(1,0)$ different from each line $L_i$ ($1 \leq i \leq k-1$).
In case $\Kar (K)=0$ applying Lemma \ref{lemma3} to the linear system $\vert (k-2,y_1) -D_1 - ... - D_{k-1} \vert$ we find $\mathcal{L}$ induces a complete linear system by restriction to $L$ (and a similar statement is obtained in case $\Kar (K)=p>0$).
This implies $\mathcal{L}$ has no base point outside of $L_1 \cup ... \cup L_{k-1}$ and if $P \in M \setminus (L_1 \cup ... \cup L_{k-1})$ then there exist $\gamma \in \mathcal {L}$ smooth at $P$ such that if $P \in L \in \vert (1,0) \vert$ then $T_P (\gamma)\neq L$.

In case $\Kar (K)=0$, because $\mathcal {L}$ has no base points outside of $D_1 + ... + D_{k-1}$ and a general element $\gamma \in \mathcal{L}$ is smooth at $P$ belonging to $D_1 + ... + D_{k-1}$, it follows from Bertini's Theorem (see e.g. \cite{ref4}, Theorem 4.1) that a general element $\gamma \in \mathcal {L}$ is smooth.
In case $\Kar (K)=p>0$ we need to use a weaker form of Bertini's Theorem.
We are going to use Theorem 1 in \cite{ref5} and therefore we need some notation.

We use the quasi-projective variety $X=M \setminus \{ D_1 \cup ... \cup D_{k-1} \}$ and for $P \in X$ let $T_{P,\mathcal{L}}$ be the intersection of all tangent spaces $T_P(\gamma)$ with $\gamma \in \mathcal {L}$ containing $P$ (of course $\gamma$ is restricted to $X$).
Let $e_{\mathcal {L}}(P)=\dim (T_{P,\mathcal{L}})$ and for $0 \leq f \leq 2$ let $X_f= \{ P \in X : e_{\mathcal {L}}(P) \geq f \}$.
From Theorem 1 in \cite{ref5} we need to prove that for $0 \leq f \leq 2$ each subvariety $Z$ of $X_f$ satisfies $\dim (Z) \leq 2-f$.

For $P \in X \cap L_i$ for some $1 \leq i \leq k-1$ we know there exists $\gamma \in \mathcal {L}$ smooth at $P$ with $T_P (\gamma) \neq L_i$.
But $\gamma ' = L_1 \cup ... \cup L_{k-1} \cup (y_1 +1 \text { general lines of type } (0,1)) \in \mathcal {L}$ and $P \in \gamma '$ with $T_P(\gamma ')=L_i$.
This proves $T_{P,\mathcal {L}}= \{ 0 \}$, hence $P \in X_0$.
Since $e_{\mathcal {L}}(P)=0$ is an open condition for $P \in X$ it implies $\dim (X_1) \leq 1$.
We also proved that for $P \in X \setminus (L_1 \cup ... \cup L_{k-1})$ a general $\gamma \in \mathcal {L}$ containing $P$ is smooth at $P$, hence $X_2 = \emptyset$.
So Theorem 1 in \cite{ref5} implies a general element $\gamma \in \mathcal{L}$ is smooth on $X$, hence also on $M$.

Now we consider the linear system $\vert (k,a) \vert$.
A divisor in $\vert (k,a) \vert$ is called of type $F_i$ ($1 \leq i \leq k-1$) if it is the sum of $\gamma \in \mathcal {L}$, the line $L_i$, all lines of type $(0,1)$ through some point of $D_1 + ... + \widehat{D_i} + ... + D_{k-1}$ (called the fixed lines of type $(0,1)$) and $a-(y_1 + ... + \widehat{y_i} + ... + y_{k-1}+y_1)$ more lines of type $(0,1)$ (called the free lines of type $(0,1)$).
Each divisor of type $F_i$ is singular at each point of $D_1 + ... + D_{k-1}$.
Choosing $\gamma$ general and also choosing the free lines of type $(0,1)$ general we have a general divisor of type $F_i$ has  an ordinary node at each point of $D_i$.
This proves that a general $\gamma \in \vert (k,a) \vert$ singular at each point of $D_1 + ... + D_{k-1}$ has ordinary nodes at those points.
Let $\mathcal {L}_s$ be the linear subsystem of $\vert (k,a) \vert$ of curves singular at each point of $D_1 + ... + D_{k-1}$.
We are going to prove that $\mathcal {L}_s$ has no base points outside of $D_1 + ... + D_{k-1}$ and in case $\Kar (K)=p>0$ and $P \in L_i \setminus D_i$ for some $1 \leq i \leq k-1$ then there exists $\Gamma \in \mathcal {L}_s$ with $\Gamma$ smooth at $P$ and $T_P(\Gamma) \neq L_i$ and in case $P \notin L_1 \cup ... \cup L_{k-1}$ then we can take $\Gamma$ smooth at $P$.

Let $P \in L_i \setminus D_i$ for some $1 \leq i \leq k-1$.
We use a divisor of type $F_j$ for some $j \neq i$.
In case $P$ belongs to some line of type $(0,1)$  through some point $Q \in D_1 + ... + D_{k-1}$ then we take $j$ such that $Q \in D_j$, otherwise we choose $j\neq i$ arbitrary.
We use $\gamma \in \mathcal{L}$ such that $P \notin \gamma$ and we also take the free lines of type $(0,1)$ not containing $P$.
In case $\Kar (K)=p>0$ it is also possible to take $\gamma \in \mathcal {L}$ smooth at $P$ with $T_P(\gamma) \neq L_i$.
This proves the properties of $\mathcal {L}_s$ for points $P \in L_1 \cup ... \cup L_{k-1}$.

In case $P \notin L_1 \cup ... \cup L_{k-1}$ we use a divisor of type $F_i$ for some $1 \leq i \leq k-1$.
In case the line of type $(0,1)$ through $P$ contains a point $Q \in D_1 + ... + D_{k-1}$ we take $i$ such that $Q \in D_i$, otherwise we take $i$ arbitrarily.
We choose $\gamma \in \mathcal {L}$ such that $P \notin \gamma$ and we take the free lines of type $(0,1)$ not containing $P$.
In case $\Kar (K)=p>0$ we also take $\gamma \in \mathcal{L}$ smooth at $P$.
Again this proves the properties of $\mathcal{L}_s$ for $P \notin L_1 \cup ... \cup L_{k-1}$.

As in the case of $\mathcal {L}$ we can use Bertini's Theorem to show that a general element $\Gamma \in \mathcal {L}_s$ is smooth outside of $D_1 + ... + D_{k-1}$.

Finally we need to show that a general element $\Gamma$ of $\mathcal {L}_s$ is irreducible.
The divisors of type $F_i$ already show $\mathcal {L}_s$ does not have a fixed component.
If a general element of $\mathcal {L}_s$ would not be irreducible then in case $\Kar (K)=0$ another theorem of Bertini implies $\mathcal {L}_s$ is composed of a pencil (it need not be a linear pencil) (see e.g. \cite{ref4} Theorem 5.3).
The divisors of type $F_i$ ($1 \leq i \leq k-1$) show this is not the case.
In case $\Kar (K)=p>0$ one also has the possibility that the general element of $\mathcal {L}_s$ is of the form $p^eU$ for some $e\geq 1$ and $U$ an irreducible  divisor (see \cite{ref6}, Section 1).
In that case all divisors are of type $p^eU$ and the divisors of type $F_i$ show this is not the case.
\end{proof}

In order to make this paper complete we give the argument giving rise to the main theorem of \cite{ref1} obtaining a slightly better bound.
The arguments are those from \cite{ref1}, part (ii) of the proof of the Main Theorem.

\begin{theorem}\label{theorem2}
Let $ e \in \mathbb{Z}_{ \geq 1}$.
In case $\Kar (K)=0$ let $A(e)=(k-1)((k-1)e-2)$ and in case $\Kar (K)=p>0$ let $A(e)=(k-1)((k-1)e-1)$.
For each sequence $0 \leq e_1 \leq ... \leq e_{k-1} = e_1+e$ such that $g=k-1+e_1 + ... + e_{k-1} > A(e)$ there exists a smooth $k$-gonal curve $C$ of genus $g$ such that a linear system $g^1_k$ on $C$ has scrollar invariants $(e_1, ..., e_{k-1})$.
\end{theorem}

\begin{proof}
Fix a sequence of integers $e=E_{k-1}\geq E_{k-2} \geq ... \geq E_2 \geq E_1=0$.
On $M = \mathbb{P}^1 \times \mathbb {P}^1$ choose $k-1$ different lines $L_1, ..., L_{k-1}$ of type $(1,0)$ and choose an effective reduced divisor $D_i$ on $L_i$ of degree $E_i$ such that no line of type $(0,1)$ contains more than one point of $D_1 \cup ... \cup D_{k-1}$.
In case $\Kar (K)=0$ let $A=(k-1)e$ and in case $\Kar (K)=p>0$ let $A=(k-1)e+1$.
From Proposition \ref{proposition1} we know that for $a \geq A$ there exists an irreducible nodal curve in $\vert (k,a) \vert$ having its nodes exactly at $D_1 \cup ... \cup D_{k-1}$.
Let $C$ be its normalization and let $g^1_k$ be the linear system on $C$ induced by the pencil $\vert (0,1) \vert$ on $M$.
From Theorem \ref{theorem1} we know the scrollar invariants of $g^1_k$ are $a-2-e, a-2-E_{k-2}, ..., a-2-E_2, a-2$.

In this way we obtain all scrollar invariants $e_1, e_2, ... , e_{k-1}$ for a $g^1_k$ satisfying $e_{k-1} - e_{k-i} = E_i$ for $2 \leq i \leq k-1$ as soon as $e_{k-1} \geq A-2$.
Varying $E_{k-2}, ... , E_2$ we obtain all scrollar invariants for a $g^1_k$ satisfying $e_{k-1}-e_1=e$ as soon as $e_{k-1} \geq A-2$.

Now take $g > (k-1)(A-2)$ and $ e_1 \leq ... \leq e_{k-1}$ with $e_1 + ... + e_{k-1}=g-k+1$ and $e_{k-1}-e_1 = e$.
In case $e_{k-1} \leq A-3$ we would obtain $g \leq (k-1)(A-2)$, a contradiction.
Hence $e_{k-1} \geq A-2$ and we found the existence of a smooth curve $C$ of genus $g$ having a $g^1_k$ with scrollar invariants $e_1 , ... , e_{k-1}$.
\end{proof}


\begin{bibsection}
\begin{biblist}

\bib{ref1}{article}{
	author={E. Ballico},
	title={scrollar invariants of smooth projective curves},
	journal={Journal of pure and applied Algebra},
	volume={166},
	year={2002},
	pages={239-246},
}
\bib{ref5}{article}{
	author={M. Coppens},
	title={Smooth hypersurfaces containing a given closed subscheme}
	journal={Comm. in Algebra},
	volume={22},
	year={1994},
	pages={5299-5311},
}
\bib{ref3}{article}{
	author={M. Coppens},
	title={The uniqueness of Weierstrass points with semigroup $(a,b)$ and related semigroups},
	journal={Abh. Math. Sem. Univ. Hamburg},
	volume={89},
	year={2019},
	pages={1-16},
}
\bib{ref2}{book}{
	author={R. Hartshorne},
	title={Algebraic Geometry},
	series={Graduate Texts in Mathematics},
	volume={52},
	year={1977},
	publisher={Springer-Verlag},
}
\bib{ref4}{article}{
	author={S. Kleiman},
	title={Bertini and his two fundamental theorems},
	journal={Rend. Circ. Mat. Palermo},
	volume={55},
	year={1997},
	pages={9-37},	
}

\bib{ref6}{article}{
	author={O. Zariski},
	title={Introduction to the problem of minimal models in the theoryy of algebraic surfaces},
	journal={Publ. Math. Soc. Japan},
	volume={4},
	year={1958},
	pages={277-369},
}


	

\end{biblist}
\end{bibsection}
\end{document}